\documentclass[12pt,a4paper,reqno]{amsart}
\title{On uniqueness of $\mathbb{P}$-twists}
\author{Rina Anno}
\email{ranno@math.ksu.edu}
\address{Department of Mathematics \\
Kansas State University \\
138 Cardwell Hall \\
Manhattan, KS 66506\\
USA}
\author{Timothy Logvinenko} 
\email{LogvinenkoT@cardiff.ac.uk} 
\address{School of Mathematics\\ 
Cardiff University\\
Senghennydd Road,\\
Cardiff, CF24 4AG\\
UK}
\usepackage{amsmath,amsfonts,amssymb,amsthm,epsfig,amscd,latexsym,comment}
\usepackage{caption}
\usepackage{tikz-cd}
\usepackage{graphicx}
\usepackage{array}
\usepackage{subfigure}
\usepackage{leftidx}
\usepackage{xparse}
\usepackage[all]{xy}
\newdir^{ (}{{}*!/-5pt/\dir^{(}}

\usepackage[colorlinks=true, pdfpagemode=none, pdfmenubar=false, linkcolor=blue, citecolor=blue, urlcolor=blue]{hyperref}

\DeclareMathOperator{\krn}{Ker}

\DeclareMathOperator{\img}{Im}

\DeclareMathOperator{\homm}{Hom}

\DeclareMathOperator{\eend}{End}

\DeclareMathOperator{\picr}{Pic}

\DeclareMathOperator{\cl}{Cl}

\DeclareMathOperator{\ext}{Ext}

\DeclareMathOperator{\trace}{tr}

\DeclareMathOperator{\action}{act}

\DeclareMathOperator{\modd}{\bf Mod}
\DeclareMathOperator{\lder}{\bf L}

\DeclareMathOperator{\ldertimes}{\overset{\lder}{\otimes}}

\DeclareMathOperator{\id}{Id}

\DeclareMathOperator{\cone}{Cone}
\DeclareMathOperator{\opp}{{opp}}
\DeclareMathOperator{\fg}{{\it fg}}
\DeclareMathOperator{\qrep}{\it \mathcal{Q}r}
\DeclareMathOperator{\hproj}{\mathcal{P}}

\DeclareMathOperator{\semifree}{\mathcal{S}\mathcal{F}}
\DeclareMathOperator{\sffg}{\mathcal{S}\mathcal{F}_{\fg}}
\DeclareMathOperator{\perf}{{\it \mathcal{P}erf}}
\DeclareMathOperator{\hmtpy}{{Ho}}

\DeclareMathOperator{\tria}{{Tria}}
\DeclareMathOperator{\pretriag}{{Pre\text{-}Tr}}

\DeclareMathOperator{\TPair}{{\bf TPair}}
\DeclareMathOperator{\alg}{{\bf Alg}}

\begin{document}

\def\bv{\mathbf{v}}
\def\kgc_{K^*_G(\mathbb{C}^n)}
\def\kgchi_{K^*_\chi(\mathbb{C}^n)}
\def\kgcf_{K_G(\mathbb{C}^n)}
\def\kgchif_{K_\chi(\mathbb{C}^n)}
\def\gpic_{G\text{-}\picr}
\def\gcl_{G\text{-}\cl}
\def\trch_{{\chi_{0}}}
\def\regring{{R}}
\def\regrep{{V_{\text{reg}}}}
\def\givrep{{V_{\text{giv}}}}
\def\lbar{{(\mathbb{Z}^n)^\vee}}
\def\genpx_{{p_X}}
\def\genpy_{{p_Y}}
\def\genpcn_{p_{\mathbb{C}^n}}
\def\gnat{gnat}
\def\twalg{{\regring \rtimes G}}
\def\L{{\mathcal{L}}}
\def\O{{\mathcal{O}}}
\def\gcd{\mbox{gcd}}
\def\lcm{\mbox{lcm}}
\def\tf{{\tilde{f}}}
\def\tD{{\tilde{D}}}
\def\A{{\mathcal{A}}}
\def\B{{\mathcal{B}}}
\def\C{{\mathcal{C}}}
\def\D{{\mathcal{D}}}
\def\F{{\mathcal{F}}}
\def\H{{\mathcal{H}}}
\def\L{{\mathcal{L}}}
\def\R{{\mathcal{R}}}
\def\barA{{\bar{\mathcal{A}}}}
\def\barAi{{\bar{\mathcal{A}}_1}}
\def\barAj{{\bar{\mathcal{A}}_2}}
\def\barB{{\bar{\mathcal{B}}}}
\def\barC{{\bar{\mathcal{C}}}}
\def\M{{\mathcal{M}}}
\def\Aopp{{\A^{\opp}}}
\def\Bopp{{\B^{\opp}}}
\def\Copp{{\C^{\opp}}}
\def\aA{\leftidx{_{a}}{\A}}
\def\bA{\leftidx{_{b}}{\A}}
\def\Aa{{\A_a}}
\def\Ea{E_a}
\def\aE{\leftidx{_{a}}{E}{}}
\def\Eb{E_b}
\def\bE{\leftidx{_{b}}{E}{}}
\def\Fa{F_a}
\def\aF{\leftidx{_{a}}{F}{}}
\def\Fb{F_b}
\def\bF{\leftidx{_{b}}{F}{}}
\def\aM{\leftidx{_{a}}{M}{}}
\def\aMb{\leftidx{_{a}}{M}{_{b}}}
\def\Ma{{M_a}}
\def\modk{{\modd\text{-}k}}
\def\modA{{\modd\text{-}\A}}
\def\modbar{{\overline{\modd}}}
\def\modbarA{{\overline{\modd}\text{-}\A}}
\def\modbarAopp{{\overline{\modd}\text{-}\Aopp}}
\def\modB{{\modd\text{-}\B}}
\def\modC{{\modd\text{-}\C}}
\def\modbarB{{\overline{\modd}\text{-}\B}}
\def\modbarBopp{{\overline{\modd}\text{-}\Bopp}}
\def\AmodA{{\A\text{-}\modd\text{-}\A}}
\def\AmodB{{\A\text{-}\modd\text{-}\B}}
\def\BmodB{{\B\text{-}\modd\text{-}\B}}
\def\BmodA{{\B\text{-}\modd\text{-}\A}}
\def\AmodbarA{\A\text{-}{\overline{\modd}\text{-}\A}}
\def\AmodbarB{\A\text{-}{\overline{\modd}\text{-}\B}}
\def\AmodbarC{\A\text{-}{\overline{\modd}\text{-}\C}}
\def\AmodbarD{\A\text{-}{\overline{\modd}\text{-}\D}}
\def\BmodbarA{\B\text{-}{\overline{\modd}\text{-}\A}}
\def\BmodbarB{\B\text{-}{\overline{\modd}\text{-}\B}}
\def\BmodbarC{\B\text{-}{\overline{\modd}\text{-}\C}}
\def\BmodbarD{\B\text{-}{\overline{\modd}\text{-}\D}}
\def\CmodbarA{\C\text{-}{\overline{\modd}\text{-}\A}}
\def\CmodbarB{\C\text{-}{\overline{\modd}\text{-}\B}}
\def\CmodbarC{\C\text{-}{\overline{\modd}\text{-}\C}}
\def\CmodbarD{\C\text{-}{\overline{\modd}\text{-}\D}}
\def\DmodbarA{\D\text{-}{\overline{\modd}\text{-}\A}}
\def\DmodbarB{\D\text{-}{\overline{\modd}\text{-}\B}}
\def\DmodbarC{\D\text{-}{\overline{\modd}\text{-}\C}}
\def\DmodbarD{\D\text{-}{\overline{\modd}\text{-}\D}}
\def\sfA{{\semifree(\A)}}
\def\sfB{{\semifree(\B)}}
\def\sffgA{{\sffg(\A)}}
\def\sffgB{{\sffg(\B)}}
\def\hprojA{{\hproj(\A)}}
\def\hprojB{{\hproj(\B)}}
\def\qrepA{{\qrep(\A)}}
\def\qrepB{{\qrep(\B)}}
\def\opp{{\text{opp}}}
\def\perfsf{{\semifree^{\perf}}}
\def\prfhpr{{\hproj^{\scriptscriptstyle\perf}}}
\def\prfhprA{{\prfhpr(\A)}}
\def\prfhprB{{\prfhpr(\B)}}
\def\prfhprAopp{{\prfhpr(\Aopp)}}
\def\prfhprBopp{{\prfhpr(\Bopp)}}
\def\perfsfA{{\perfsf(\A)}}
\def\perfsfB{{\perfsf(\B)}}
\def\qrhpr{{\hproj^{qr}}}
\def\qrhprA{{\qrhpr(\A)}}
\def\qrhprB{{\qrhpr(\B)}}
\def\qrsf{{\semifree^{qr}}}
\def\qrsf{{\semifree^{qr}}}
\def\qrsfA{{\qrsf(\A)}}
\def\qrsfB{{\qrsf(\B)}}
\def\Aperfsf{{\semifree^{\A\text{-}\perf}(\AbimB)}}
\def\Bperfsf{{\semifree^{\B\text{-}\perf}(\AbimB)}}
\def\Aprfhpr{{\hproj^{\A\text{-}\perf}(\AbimB)}}
\def\Bprfhpr{{\hproj^{\B\text{-}\perf}(\AbimB)}}
\def\Aqrhpr{{\hproj^{\A\text{-}qr}(\AbimB)}}
\def\Bqrhpr{{\hproj^{\B\text{-}qr}(\AbimB)}}
\def\Aqrsf{{\semifree^{\A\text{-}qr}(\AbimB)}}
\def\Bqrsf{{\semifree^{\B\text{-}qr}(\AbimB)}}
\def\modAopp{{\modd\text{-}\Aopp}}
\def\modBopp{{\modd\text{-}\Bopp}}
\def\AmodA{{\A\text{-}\modd\text{-}\A}}
\def\AmodB{{\A\text{-}\modd\text{-}\B}}
\def\AmodC{{\A\text{-}\modd\text{-}\C}}
\def\BmodA{{\B\text{-}\modd\text{-}\A}}
\def\BmodB{{\B\text{-}\modd\text{-}\B}}
\def\BmodC{{\B\text{-}\modd\text{-}\C}}
\def\CmodA{{\C\text{-}\modd\text{-}\A}}
\def\CmodB{{\C\text{-}\modd\text{-}\B}}
\def\CmodC{{\C\text{-}\modd\text{-}\C}}
\def\AbimA{{\A\text{-}\A}}
\def\AbimC{{\A\text{-}\C}}
\def\BbimA{{\B\text{-}\A}}
\def\BbimB{{\B\text{-}\B}}
\def\BbimC{{\B\text{-}\C}}
\def\CbimA{{\C\text{-}\A}}
\def\CbimB{{\C\text{-}\B}}
\def\CbimC{{\C\text{-}\C}}
\def\AhprA{{\hproj\left(\AbimA\right)}}
\def\BhprB{{\hproj\left(\BbimB\right)}}
\def\AhprB{{\hproj\left(\AbimB\right)}}
\def\BhprA{{\hproj\left(\BbimA\right)}}
\def\AbarA{{\overline{\A\text{-}\A}}}
\def\AbarB{{\overline{\A\text{-}\B}}}
\def\BbarA{{\overline{\B\text{-}\A}}}
\def\BbarB{{\overline{\B\text{-}\B}}}
\def\QAbimB{{Q\A\text{-}\B}}
\def\AbimB{{\A\text{-}\B}}
\def\AonebimB{{\A_1\text{-}\B}}
\def\AtwobimB{{\A_2\text{-}\B}}
\def\BbimA{{\B\text{-}\A}}
\def\Aperf{{\A\text{-}\perf}}
\def\Bperf{{\B\text{-}\perf}}
\def\MddA{{M^{\tilde{\A}}}}
\def\MddB{{M^{\tilde{\B}}}}
\def\MhdA{{M^{h\A}}}
\def\MhdB{{M^{h\B}}}
\def\NhdB{{N^{h\B}}}
\def\Cat{{Cat}}
\def\DGCat{{DG\text{-}Cat}}
\def\HoDGCat{{\hmtpy(\DGCat)}}
\def\HoDGCatV{{\hmtpy(\DGCat_\mathbb{V})}}
\def\tr{{tr}}
\def\pretr{{pretr}}
\def\kctr{{kctr}}
\def\PreTrCat{{\DGCat^\pretr}}
\def\KcTrCat{{\DGCat^\kctr}}
\def\HoPretrCat{{\hmtpy(\PreTrCat)}}
\def\HoKcTrCat{{\hmtpy(\KcTrCat)}}
\def\Aquasirep{{\A\text{-}qr}}
\def\QAquasirep{{Q\A\text{-}qr}}
\def\Bquasirep{{\B\text{-}qr}} 
\def\lderA{{\tilde{\A}}} 
\def\lderB{{\tilde{\B}}} 
\def\adjunit{{\text{adj.unit}}}
\def\adjcounit{{\text{adj.counit}}}
\def\degzero{{\text{deg.0}}}
\def\degone{{\text{deg.1}}}
\def\degminusone{{\text{deg.-$1$}}}
\def\bareta{{\overline{\eta}}}
\def\barzeta{{\overline{\zeta}}}
\def\Ract{{R {\action}}}
\def\barRact{{\overline{\Ract}}}
\def\actL{{{\action} L}}
\def\baractL{{\overline{\actL}}}
\def\Ainfty{{A_{\infty}}}
\def\noddinf{{{\bf Nod}_{\infty}}}
\def\noddinfstr{{{\bf Nod}^{\text{strict}}_{\infty}}}
\def\noddinfA{{\noddinf\A}}
\def\noddinfB{{\noddinf\B}}
\def\noddinfAB{{\noddinf\AbimB}}
\def\noddinfBA{{\noddinf\BbimA}}
\def\noddinfu{{({\bf Nod}_{\infty})_u}}
\def\noddinfuA{{(\noddinfA)_u}}
\def\noddinfhu{{({\bf Nod}_{\infty})_{hu}}}
\def\noddinfhuA{{(\noddinfA)_{hu}}}
\def\noddinfdg{{({\bf Nod}_{\infty})_{dg}}}
\def\noddinfdgA{{(\noddinfA)_{dg}}}
\def\noddinfdgAA{{(\noddinf\AbimA)_{dg}}}
\def\noddinfdgAB{{(\noddinf\AbimB)_{dg}}}
\def\noddinfdgB{{(\noddinfB)_{dg}}}
\def\moddinf{{\modd_{\infty}}}
\def\moddinfA{{\modd_{\infty}\A}}
\def\infbar{{B_\infty}}
\def\infbarA{{B^\A_\infty}}
\def\infbarB{{B^\B_\infty}}
\def\infbarC{{B^\C_\infty}}
\def\inftimes{{\overset{\infty}{\otimes}}}
\def\infhom{{\overset{\infty}{\homm}}}
\def\barhom{{\overline{\homm}}}
\def\barend{{\overline{\eend}}}
\def\bartimes{{\;\overline{\otimes}}}
\def\triaA{{\tria \A}}
\def\TPairdg{{\TPair^{dg}}}
\def\algA{{\alg(\A)}}
\def\Ainfty{{A_{\infty}}}

\theoremstyle{definition}
\newtheorem{defn}{Definition}[section]
\newtheorem*{defn*}{Definition}
\newtheorem{exmpl}[defn]{Example}
\newtheorem*{exmpl*}{Example}
\newtheorem{exrc}[defn]{Exercise}
\newtheorem*{exrc*}{Exercise}
\newtheorem*{chk*}{Check}
\newtheorem*{remarks*}{Remarks}
\theoremstyle{plain}
\newtheorem{theorem}{Theorem}[section]
\newtheorem*{theorem*}{Theorem}
\newtheorem{conj}[defn]{Conjecture}
\newtheorem*{conj*}{Conjecture}
\newtheorem{prps}[defn]{Proposition}
\newtheorem*{prps*}{Proposition}
\newtheorem{cor}[defn]{Corollary}
\newtheorem*{cor*}{Corollary}
\newtheorem{lemma}[defn]{Lemma}
\newtheorem*{claim*}{Claim}
\newtheorem{Specialthm}{Theorem}
\renewcommand\theSpecialthm{\Alph{Specialthm}}
\numberwithin{equation}{section}
\renewcommand{\textfraction}{0.001}
\renewcommand{\topfraction}{0.999}
\renewcommand{\bottomfraction}{0.999}
\renewcommand{\floatpagefraction}{0.9}
\setlength{\textfloatsep}{5pt}
\setlength{\floatsep}{0pt}
\setlength{\abovecaptionskip}{2pt}
\setlength{\belowcaptionskip}{2pt}
\begin{abstract}
We prove that for any $\mathbb{P}^n$-functor all the convolutions 
(double cones) of the three-term complex 
$FHR \xrightarrow{\psi} FR \xrightarrow{\trace} \id$ 
defining its $\mathbb{P}$-twist are isomorphic. We also introduce 
a new notion of a non-split $\mathbb{P}^n$-functor. 
\end{abstract}

\maketitle

\section{Introduction}

A \em $\mathbb{P}^n$-object \rm $E$ in the derived category $D(X)$ of a smooth projective
variety $X$ has $\ext^*_{X}(E,E) \simeq H^*(\mathbb{P}^n)$
as graded rings and $E \otimes \omega_X \simeq E$. These were introduced by Huybrechts and Thomas
in \cite{HuybrechtsThomas-PnObjectsAndAutoequivalencesOfDerivedCategories}
as mirror symmetric analogues of Lagrangian $\mathbb{C}\mathbb{P}^n$s
in a Calabi Yau manifold. Moreover, there is an analogue of the Dehn twist: the 
\em $\mathbb{P}$-twist \rm $P_E$ about $E$ is the Fourier-Mukai
transform defined by a certain convolution (double cone) of the three term 
complex
\begin{align}
\label{eqn-intro-Pn-object-three-term-complex}	
E^\vee \boxtimes E [-2] \xrightarrow{h^\vee \otimes \id - \id \otimes h} 
E^\vee \boxtimes E \xrightarrow{\trace} O_{\Delta}
\end{align}
where $h$ is the degree $2$ generator of  $\ext^*_{X}(E,E)$.
It was shown in 
\cite{HuybrechtsThomas-PnObjectsAndAutoequivalencesOfDerivedCategories}
to be an auto-equivalence of $D(X)$.

A \em convolution \rm of a three term complex in a triangulated category
$\mathcal{D}$
\begin{align}
A \xrightarrow{f} B \xrightarrow{g} C	
\end{align}
is any object obtained via one of the following two constructions. 
A \em left Postnikov system \rm is where we first take the cone $Y$ of
$f$, then lift $g$ to a morphism $m\colon Y \rightarrow C$, and 
take the cone of $m$.  A \em right Postnikov system \rm is where 
we first take cone $X$ of $g$, then lift $f$ to a morphism 
$j \colon A[1] \rightarrow X$, and take a cone of $j$. 
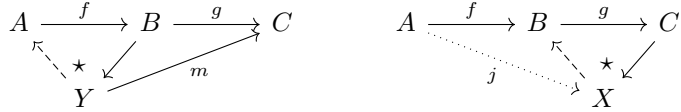
\begin{equation*}
\begin{tikzcd}[column sep={0.866cm,between origins}, row sep={1cm,between origins}]
A
\ar{rr}{f}
& &
B
\ar{rr}{g}
\ar{dl}
\ar[phantom]{dll}[description, pos=0.475]{\star}
& &
C
\\
~
& 
\ar[dashed]{ul}
Y
\ar{urrr}[']{m}
&
&
&
\end{tikzcd}
\quad \quad \quad 
\begin{tikzcd}[column sep={0.866cm,between origins}, row sep={1cm,between origins}]
A
\ar{rr}{f}
\ar[dotted]{drrr}[']{j}
& &
B
\ar{rr}{g}
\ar[phantom]{drr}[description, pos=0.45]{\star}
& &
C
\ar{dl}
\\
& &
&
X
\ar[dashed]{ul}
&
~
\end{tikzcd}
\end{equation*}
Apriori, convolutions are not unique. For example, the convolutions
of $A[-2] \rightarrow 0 \rightarrow B$ are  
the extensions of $A$ by $B$ in $\mathcal{D}$. If $\mathcal{D}$
admits a DG-enhancement $\mathcal{C}$, the convolutions of a complex 
in $\mathcal{D}$ up to isomorphism are in bijection with 
the \em twisted complex \rm structures on it in $\mathcal{C}$ up to
homotopy equivalence, cf. \S\ref{section-twisted-complexes}. 

In \cite[Lemma
2.1]{HuybrechtsThomas-PnObjectsAndAutoequivalencesOfDerivedCategories}
it was shown that 
the complex \eqref{eqn-intro-Pn-object-three-term-complex} has 
a unique left Postnikov system and defined the $\mathbb{P}$-twist 
to be its convolution. Later Addington noted in
\cite{Addington-NewDerivedSymmetriesOfSomeHyperkaehlerVarieties}
that 
$$\homm^{-1}_{D(X\times X)}(E^\vee \boxtimes E [-2], \mathcal{O}_\Delta) 
\simeq \homm^1_X(E,E) = 0$$ which by a simple homological argument 
implies that the complex \eqref{eqn-intro-Pn-object-three-term-complex} has 
a unique convolution. See Lemma 
\ref{lemma-uniqueness-of-convolutions-for-a-two-step-complex} of this
note and the first remark after it. 

In \cite{Addington-NewDerivedSymmetriesOfSomeHyperkaehlerVarieties}
and \cite{Cautis-FlopsAndAboutAGuide} Addington and Cautis introduced
the notion of a \em (split) $\mathbb{P}^n$-functor \rm to  
generalise $\mathbb{P}^n$-objects in a similar way to 
spherical functors \cite{AnnoLogvinenko-SphericalDGFunctors} 
generalising spherical objects 
\cite{SeidelThomas-BraidGroupActionsOnDerivedCategoriesOfCoherentSheaves}. 
It was a brilliant idea and numerous applications followed
\cite{Krug-NewDerivedAutoequivalencesOfHilbertSchemesAndGeneralisedKummerVarieties}\cite{Krug-P-functorVersionsOfTheNakajimaOperators}
\cite{AddingtonDonovanMeachan-ModuliSpacesOfTorsionSheavesOnK3SurfacesAndDerivedEquivalences}\cite{AddingtonDonovanMeachan-MukaiFlopsAndPTwists}. 

For $Z$ and $X$ smooth projective varieties  
a \em split $\mathbb{P}^n$-functor \rm is a Fourier-Mukai functor 
$F\colon D(Z) \rightarrow D(X)$ with left and right Fourier-Mukai 
adjoints $L,R$ such that for some autoequivalence $H$ of $D(Z)$  
we have an isomorphism
\begin{align}
\label{eqn-intro-RF-for-split-Pn-functor}
RF \simeq H^n \oplus H^{n-1} \oplus \dots \oplus H \oplus \id 
\end{align}
satisfying the \em monad condition \rm
and the \em adjoints condition \rm which
generalise the 
$\mathbb{P}$-object requirements of 
$\ext^*_{X}(E,E) \simeq H^*(\mathbb{P}^n)$ respecting
the graded ring structure 
and of $E \simeq E \otimes \omega_X$. 
The \em $\mathbb{P}$-twist \rm about $F$ is then the convolution of 
a certain canonical right Postnikov system of the three-term complex
\begin{align}
\label{eqn-intro-Pn-functor-three-term-complex}
FHR \xrightarrow{\psi} FR \xrightarrow{\trace} \id 
\end{align}
where $\trace$ is the adjunction co-unit and $\psi$ the corresponding
component of the map $FRFR \xrightarrow{FR\trace - \trace FR} FR$ 
after the identification 
\eqref{eqn-intro-RF-for-split-Pn-functor}. 

Addington noted in
\cite[\S4.3]{Addington-NewDerivedSymmetriesOfSomeHyperkaehlerVarieties}
that Postnikov systems for \eqref{eqn-intro-Pn-object-three-term-complex}
are not necessarily unique. This caused many
technical difficulties in applications. They were aggravated
by the fact that it was sometimes simpler to calculate 
left Postnikov systems associated to 
\eqref{eqn-intro-Pn-object-three-term-complex}.  
In a word, it was often easy to compute some convolution of 
\eqref{eqn-intro-Pn-object-three-term-complex} but difficult 
to prove that it was indeed the $\mathbb{P}$-twist defined in 
\cite{Addington-NewDerivedSymmetriesOfSomeHyperkaehlerVarieties}. 

The main result of this paper is that contrary to the expectations
of specialists, including  the authors of this paper, 
the three term complex \eqref{eqn-intro-Pn-object-three-term-complex}
has a unique convolution. Thus we can compute the $\mathbb{P}$-twist
via any Postnikov system, taking cones in any order and using
any lifts. To prove this we prove a more general fact: 
\begin{theorem*}
[see Theorem \ref{theorem-uniqueness-of-FG-FR-Id-convolutions-triangulated-version}]
Let $Z,X$ be separated schemes of finite type over a field. 
Let $F\colon D(Z) \rightarrow D(X)$ be an exact functor  
with a right adjoint $R$. Let $\trace\colon 
FR \rightarrow \id_X$ be the adjunction co-unit. 
Let $G\colon D(X) \rightarrow D(Z)$ be any exact functor and 
$f\colon FG \rightarrow FR$ any natural transformation with $\trace \circ f = 0$. 
Finally, assume these are all Fourier-Mukai functors and natural 
transformations thereof. 

Then all convolutions of the following three-term complex are isomorphic:
\begin{align}
\label{eqn-intro-FG-FR-Id-complex}
FG \xrightarrow{f} FR \xrightarrow{\trace} \id_X. 
\end{align}
\end{theorem*}

Our proof shows that
the complex \eqref{eqn-intro-FG-FR-Id-complex} has a unique right Postnikov 
system. We then prove in Lemma
\ref{lemma-left-Postnikov-system-for-each-right-Postnikov-system-and-vice-versa}
that for any left Postnikov system there exists a right 
Postnikov system with the same convolution, and vice versa. 
We give these proofs purely in the language of triangulated
categories. 

In a DG-enhanced setting we can work more generally and 
give a more direct proof. 
In Prop.~\ref{prps-any-two-lifts-of-XM-NM-B-are-isomorphic}
we construct a homotopy equivalence between any two twisted 
complex structures on \eqref{eqn-intro-FG-FR-Id-complex}. We thus obtain:
\begin{theorem*}[see Theorem
\ref{theorem-uniqueness-of-FG-FR-Id-convolutions-dg-version}]

Let $\A$ and $\B$ be enhanced triangulated categories. Let 
$F\colon \A \rightarrow \B$ be an exact functor with a right adjoint $R$. 
Let $\trace\colon FR \rightarrow \id_\B$ be the adjunction counit. 
Let $G\colon \B \rightarrow \A$ be any exact functor and  
$f\colon FG \rightarrow FR$ any natural transformation 
with $\trace \circ f = 0$. Finally, assume these are all DG-enhanceable.

Then all convolutions of the following three-term complex are isomorphic:
\begin{align}
\label{eqn-intro-ordinary-complex-FG-FR-Id-enhanced-setting}
FG \xrightarrow{f} FR \xrightarrow{\trace} \id_\B. 
\end{align}
\end{theorem*}

The uniqueness of $\mathbb{P}$-twists as established by these two
theorems removes a significant roadblock in the way of research 
into $\mathbb{P}^n$-functors. Our results were immediately applied 
in a number of papers including
\cite{HocheneggerKrug-FormalityOfPObjects}, 
\cite{KrugMeachan-UniversalFunctorsOnSymmetricQuotientStacksOfAbelianVarieties},
\cite{MeachanRaedschelders-HochschildCohomologyAndDeformationsOfPFunctors}. 

Finally, Addington and Cautis referred to the notion which they
introduced as $\mathbb{P}^n$-functors. The reason it is perhaps
best referred to as \em split \rm $\mathbb{P}^n$-functors is that
the monad $RF$ splits into a direct sum of $\id$ and powers of an
autoequivalence $H$. In the definition of a spherical functor $RF$ 
can be a non-trivial extension of $\id$ by an autoequivalence, and this
is the case in many interesting examples. Indeed, it was later 
noted by Addington, Donovan, and Meachan in \cite[Remark
1.7]{AddingtonDonovanMeachan-MukaiFlopsAndPTwists} that it would 
be nice to allow $RF$ to have a filtration with quotients
$\id, H, \dots, H^n$, however it would then be difficult to formulate
the monad condition and to construct $\mathbb{P}^n$-twist as a
convolution of a three-term complex. 

In \S\ref{section-p-functors} we propose a general notion of 
a (non-split) \em $\mathbb{P}^n$-functor \rm which deals with all 
of these issues. These are the functors $F$ for which $RF$ is 
isomorphic to a repeated extension of $\id$ by $H, \dots, H^n$ of the form
\begin{footnotesize}
\begin{equation*}
\begin{tikzcd}[column sep = {0.35cm}]
\id 
\ar[phantom]{drr}[description, pos=0.45]{\star}
\ar{rr}{\iota_1}
& &
Q_1 
\ar{rr}{\iota_2}
\ar{ld}{\mu_1}
\ar[phantom]{drr}[description, pos=0.45]{\star}
& &
Q_2  
\ar{r}
\ar{ld}{\mu_2}
&
\dots 
\ar{r}
&
Q_{n-2}
\ar[phantom]{drr}[description, pos=0.45]{\star}
\ar{rr}{\iota_{n-1}}
\ar{ld}
& 
& 
Q_{n-1}
\ar{ld}{\mu_{n-1}}
\ar{rr}{\iota_{n}}
& & 
Q_n.
\ar[phantom]{dll}[description, pos=0.45]{\star}
\ar{ld}{\mu_n}
\\
&
H
\ar[dashed]{lu}{\sigma}
& ~ &
H^2 
\ar[dashed]{lu}
& ~ & 
\dots
\ar[dashed]{lu}
&
&
H^{n-1}
\ar[dashed]{lu}
&~& 
H^n
\ar[dashed]{lu}
& 
\end{tikzcd}
\end{equation*}
\end{footnotesize}
This has to satisfy three conditions: the monad condition, the adjoints condition, 
and the highest degree term condition, see \S\ref{section-p-functors}.
In comparison, the definition in
\cite{Addington-NewDerivedSymmetriesOfSomeHyperkaehlerVarieties} only
asks for two conditions. However, in the non-split situation, the
precise analogue of the monad condition in
\cite{Addington-NewDerivedSymmetriesOfSomeHyperkaehlerVarieties} is
complicated to state on the level of triangulated categories. We weaken it to 
the point where it can be easily stated on the triangulated level, but at the price of introducing the highest degree term condition. However, as explained in \S\ref{section-p-functors}, if the non-split analogues of the two conditions in \cite{Addington-NewDerivedSymmetriesOfSomeHyperkaehlerVarieties} hold, they do imply our three conditions. Thus our definition is strictly more general. 

We define the \em $\mathbb{P}$-twist \rm about such $F$ 
to be the unique convolution of the three-term complex 
\begin{align}
\label{eqn-intro-ordinary-complex-FG-FR-Id-nonsplit-setting}
FHR \xrightarrow{\psi} FR \xrightarrow{\trace} \id_\B 
\end{align}
where $\psi$ is again the corresponding component of 
$FR\trace - \trace FR$ after appropriate
identifications. The uniqueness follows by  
Theorem \ref{theorem-uniqueness-of-FG-FR-Id-convolutions-dg-version}
of this paper. In \cite{AnnoLogvinenko-PFunctors} 
we show that this $\mathbb{P}$-twist is indeed an autoequivalence
and give examples of non-split $\mathbb{P}^n$-functors.   

On the structure of this paper. 
In \S\ref{section-postnikov-systems-and-convolutions} 
and
\S\ref{section-twisted-complexes}
we give preliminaries on Postnikov systems and 
on twisted complexes, respectively. In \S\ref{section-p-functors}
we give the definition of a (non-split) $\mathbb{P}^n$-functor. Then
in \S\ref{section-an-approach-via-triangulated-categories}
and 
\S\ref{section-an-approach-via-DG-enhancements} we prove
our main results via triangulated and DG-categorical techniques,
respectively. Those only interested in the triangulated approach
should read \S\ref{section-postnikov-systems-and-convolutions},
\S\ref{section-an-approach-via-triangulated-categories}, 
and, possibly, \S\ref{section-p-functors}. 

\em Acknowledgements: \rm We would like to thank Alexei Bondal 
and Mikhail Kapranov for introducing the notions of 
a DG-enhancement and a twisted complex
in \cite{BondalKapranov-EnhancedTriangulatedCategories}. 

\section{Preliminaries}
\subsection{Postnikov systems and convolutions}
\label{section-postnikov-systems-and-convolutions}

Let $\mathcal{D}$ be a triangulated category and let 
\begin{align}
\label{eqn-A-B-C-complex}
A \xrightarrow{f} B \xrightarrow{g} C 
\end{align}
be a complex of objects of $\mathcal{D}$, that is $g \circ f = 0$. 

A \em right Postnikov system \rm associated to the complex 
\eqref{eqn-A-B-C-complex} is a diagram 
\begin{equation}
\label{eqn-right-Postnikov-system-on-A-B-C}
\begin{tikzcd}[column sep={1.155cm,between origins}, row sep={1.333cm,between origins}]
A
\ar{rr}{f}
\ar[dotted]{drrr}[']{j}
& &
B
\ar{rr}{g}
\ar[phantom]{drr}[description, pos=0.45]{\star}
& &
C
\ar{dl}{h}
\\
& &
&
X
\ar[dashed]{ul}[']{i}
&
~
\end{tikzcd}
\end{equation}
where the starred triangle is exact and the other triangle is
commutative. The dashed and dotted arrows denote maps of degree $1$
and $-1$ respectively. The \em convolution \rm of 
\eqref{eqn-right-Postnikov-system-on-A-B-C}
is the cone of the map $A[1] \xrightarrow{j} X$.

Similarly, a \em left Postnikov system \rm associated to the complex    
\eqref{eqn-A-B-C-complex} is a diagram 
\begin{equation}
\label{eqn-left-Postnikov-system-on-A-B-C}
\begin{tikzcd}[column sep={1.155cm,between origins}, row sep={1.333cm,between origins}]
A
\ar{rr}{f}
& &
B
\ar{rr}{g}
\ar{dl}[']{k}
\ar[phantom]{dll}[description, pos=0.475]{\star}
& &
C
\\
~
& 
\ar[dashed]{ul}{l}
Y
\ar{urrr}[']{m}
&
&
&
\end{tikzcd}
\end{equation}
Its \em convolution \rm is the cone of the map $Y \xrightarrow{m} C$. 

We say that an object $E \in \mathcal{D}$ is a \em convolution \rm of 
the complex \eqref{eqn-A-B-C-complex} if it is
a convolution of some right or left Postnikov system associated to it. 
The following is a direct proof of the three-term complex case of
a more general fact whose proof is sketched out in
\cite[\S{IV.2}, Exercise 1]{GelfandManin-MethodsOfHomologicalAlgebra}:

\begin{lemma}
\label{lemma-left-Postnikov-system-for-each-right-Postnikov-system-and-vice-versa}
For every right (resp. left) Postnikov system associated to the complex  
\eqref{eqn-A-B-C-complex} there is a left (resp. right) Postnikov system 
with an isomorphic convolution. 
\end{lemma}

\begin{proof}
Let 
\begin{equation}
\label{eqn-lemma-right-Postnikov-system-on-A-B-C}
\begin{tikzcd}[column sep={1.155cm,between origins}, row sep={1.333cm,between origins}]
A
\ar{rr}{f}
\ar[dotted]{drrr}[']{j}
& &
B
\ar{rr}{g}
\ar[phantom]{drr}[description, pos=0.45]{\star}
& &
C
\ar{dl}{h}
\\
& &
&
X
\ar[dashed]{ul}[']{i}
&
~
\end{tikzcd}
\end{equation}
be any right Postnikov system associated to \eqref{eqn-A-B-C-complex}. 
Then we have a commutative diagram
\begin{equation}
\label{eqn-right-postnikov-system-commutative-square}
\begin{tikzcd}
A
\ar{r}{f}
\ar{d}[']{j[-1]}
&
B
\ar[equals]{d}
\\
X[-1]
\ar{r}[']{i}
&
B.
\end{tikzcd}
\end{equation}
Let 
$$ A \xrightarrow{f} B \xrightarrow{k} Y \xrightarrow{l} A[1] $$
be any exact triangle incorporating the map $f$. 
By \cite[Lemma 2.6]{May-TheAdditivityofTracesInTriangulatedCategories}
it follows from the octahedral axiom that 
\eqref{eqn-right-postnikov-system-commutative-square} can be completed to 
the following $3 \times 3$ diagram with exact rows and columns
\begin{equation}
\label{eqn-right-postnikov-system-commutative-square-completed}
\begin{tikzcd}
A
\ar{r}{f}
\ar{d}[']{j[-1]}
&
B
\ar[equals]{d}
\ar{r}{k}
& 
Y
\ar{d}
\\
X[-1]
\ar{r}[']{i}
\ar{d}
&
B
\ar{d}
\ar{r}{g}
&
C 
\ar{d}
\\
\cone(j)[-1]
\ar{r}
& 
\cone(\id)
\ar{r}
&
Z.
\end{tikzcd}
\end{equation}
Let $m$ be the map $Y \rightarrow C$ in the right column of 
\eqref{eqn-right-postnikov-system-commutative-square-completed}. 
Since the top right square in 
\eqref{eqn-right-postnikov-system-commutative-square-completed}
commutes 
\begin{equation}
\label{eqn-lemma-left-Postnikov-system-on-A-B-C}
\begin{tikzcd}[column sep={1.155cm,between origins}, row sep={1.333cm,between origins}]
A
\ar{rr}{f}
& &
B
\ar{rr}{g}
\ar{dl}[']{k}
\ar[phantom]{dll}[description, pos=0.475]{\star}
& &
C
\\
~
& 
\ar[dashed]{ul}{l}
Y
\ar{urrr}[']{m}
&
&
&
\end{tikzcd}
\end{equation}
is a left Postnikov system associated to  \eqref{eqn-A-B-C-complex}. 
Since $\cone(\id) \simeq 0$ and the bottom row is exact the object
$Z$ is isomorphic to $\cone(j)$, i.e. the convolution of the
right Postnikov system \eqref{eqn-lemma-right-Postnikov-system-on-A-B-C}. 
On the other hand, since the right column is exact, $Z$ is isomorphic
to $\cone(m)$, i.e. the convolution of the left Postnikov system 
\eqref{eqn-lemma-left-Postnikov-system-on-A-B-C}. Thus 
\eqref{eqn-lemma-left-Postnikov-system-on-A-B-C} is a left Postnikov
system whose convolution is isomorphic to that of 
\eqref{eqn-lemma-right-Postnikov-system-on-A-B-C}, as desired. 

The proof that given a left Postnikov system associated to 
\eqref{eqn-A-B-C-complex} we can construct a right Postnikov system 
with an isomorphic convolution is analogous. 
\end{proof}

\begin{lemma}
\label{lemma-uniqueness-of-convolutions-for-a-two-step-complex}
If the natural map 
\begin{align}
\label{eqn-the-map-controlling-the-right-Postnikov-system-convolutions} 
\homm^{-1}(A,B) \xrightarrow{ g \circ (-)} \homm^{-1}(A,C)
\end{align}
is surjective then the convolutions of all right Postnikov systems
associated to \eqref{eqn-A-B-C-complex} are isomorphic. 

Similarly, if the natural map 
\begin{align}
\label{eqn-the-map-controlling-the-left-Postnikov-system-convolutions} 
\homm^{-1}(B,C) \xrightarrow{(-) \circ f} \homm^{-1}(A,C)
\end{align} 
is surjective then the convolutions of all left Postnikov systems
associated to \eqref{eqn-A-B-C-complex} are isomorphic. 
\end{lemma}
\begin{proof}
We only prove the first assertion as the second assertion is proved
similarly.  
Take any exact triangle incorporating the map $g$
\begin{align}
\label{eqn-exact-triangle-incorporating-g}
B \xrightarrow{g} C \xrightarrow{h} X \xrightarrow{i} B[1]. 
\end{align}
Then for every right Postnikov system 
\begin{equation}
\label{eqn-lemma2-another-right-Postnikov-system-on-A-B-C}
\begin{tikzcd}[column sep={1.155cm,between origins}, row sep={1.333cm,between origins}]
A
\ar{rr}{f}
\ar[dotted]{drrr}[']{j'}
& &
B
\ar{rr}{g}
\ar[phantom]{drr}[description, pos=0.45]{\star}
& &
C
\ar{dl}{h'}
\\
& &
&
X'
\ar[dashed]{ul}[']{i'}
&
~
\end{tikzcd}
\end{equation}
there exists a map $A[1] \xrightarrow{j} X$ such that 
\begin{equation}
\label{eqn-lemma2-right-Postnikov-system-on-A-B-C}
\begin{tikzcd}[column sep={1.155cm,between origins}, row sep={1.333cm,between origins}]
A
\ar{rr}{f}
\ar[dotted]{drrr}[']{j}
& &
B
\ar{rr}{g}
\ar[phantom]{drr}[description, pos=0.45]{\star}
& &
C
\ar{dl}{h}
\\
& &
&
X
\ar[dashed]{ul}[']{i}
&
~
\end{tikzcd}
\end{equation}
is a Postnikov system whose convolution is isomorphic to that 
of \eqref{eqn-lemma2-another-right-Postnikov-system-on-A-B-C}. 
Indeed, let $X' \xrightarrow{t} X$ be an isomorphism which 
identifies the exact triangles in 
\eqref{eqn-lemma2-right-Postnikov-system-on-A-B-C}
and in
\eqref{eqn-lemma2-another-right-Postnikov-system-on-A-B-C}
and set $j = t \circ j'$. 

Thus the convolutions of all right Postnikov systems
associated to \eqref{eqn-A-B-C-complex} are isomorphic
to the cones of all possible maps $A[1] \xrightarrow{j} X$
with $f = i \circ j$. In particular, to show that all the
convolutions are isomorphic it would suffice to show that that
there exists a unique map $A[1] \xrightarrow{j} X$ with 
$f = i \circ j$. 

Now consider the following fragment of the long exact sequence
obtained by applying $\homm^\bullet_{\mathcal{D}}(A,-)$ to the exact
triangle \eqref{eqn-exact-triangle-incorporating-g}:
\begin{align*}
\dots \rightarrow
\homm^{-1}_\mathcal{D}(A,B) \xrightarrow{g \circ (-)}
& \homm^{-1}_\mathcal{D}(A,C) \xrightarrow{h \circ (-)} \\
\rightarrow
& \homm^{-1}_\mathcal{D}(A,X) \xrightarrow{i \circ (-)}
\homm^{0}_\mathcal{D}(A,B) 
\rightarrow \dots
\end{align*}
Let $J$ be the set of maps $A[1] \xrightarrow{j} X$ with $f = i \circ
j$. It is the pre-image in 
$\homm^{-1}_\mathcal{D}(A,X)$
of $f \in \homm^{0}_\mathcal{D}(A,B)$. Choosing any $j_0 \in J$
induces a one-to-one correspondence between $J$ and 
$\ker \bigl(i \circ (-)\bigr) \subseteq \homm^{-1}_\mathcal{D}(A,X)$.
Thus it suffices to show that $\ker \bigl(i \circ (-)\bigr) = 0$. 
By the exactness of the fragment above 
$$\ker \bigl(h \circ (-)\bigr) = \img \bigl(h \circ (-)\bigr)$$
and $\img \bigl(h \circ (-)\bigr) = 0$ is equivalent to 
$\ker \bigl(h \circ (-)\bigr) =  \homm^{-1}_\mathcal{D}(A,C)$. 
this is further equivalent
to 
$$ \img \bigl(g \circ (-)\bigr) = \homm^{-1}_\mathcal{D}(A,C),$$
i.e. to $g \circ (-)$ being surjective, as desired. 
\end{proof}
Remarks:
\begin{enumerate}
\item Note, in particular, that if $\homm^{-1}_{\mathcal{D}}(A,C)$ is zero then 
both the criteria
in Lemma \ref{lemma-uniqueness-of-convolutions-for-a-two-step-complex}
above are automatically fulfilled. Thus these criteria each refine that
of $\homm^{-1}_{\mathcal{D}}(A,C)$ vanishing. 
\item 
In view of Lemma 
\ref{lemma-left-Postnikov-system-for-each-right-Postnikov-system-and-vice-versa}
if either of the criteria in 
Lemma \ref{lemma-uniqueness-of-convolutions-for-a-two-step-complex}
holds then the convolutions of all right and all left Postnikov
systems associated to \eqref{eqn-A-B-C-complex} are isomorphic. 
\end{enumerate}

\subsection{Twisted complexes}
\label{section-twisted-complexes}

For technical details on twisted complexes, pretriangulated categories and 
DG-enhancements see \cite[\S3]{AnnoLogvinenko-SphericalDGFunctors},
\cite{BondalKapranov-EnhancedTriangulatedCategories},
\cite[\S1]{LuntsOrlov-UniquenessOfEnhancementForTriangulatedCategories}. 

Let $\mathcal{C}$ be a $DG$-category. Let 
$\pretriag(\mathcal{C})$ be the DG-category of \em one-sided twisted
complexes \rm $(E_i, q_{ij})$ over $\mathcal{C}$. 
The category $H^0(\pretriag(\mathcal{C}))$ has a natural triangulated
structure: it is the triangulated hull of the image of 
$H^0(\C)$ in the triangulated category $H^0(\modC)$ under Yoneda
embedding. 
Throughout this section assume further that $\mathcal{C}$ is 
\em pretriangulated\rm.  
Then $H^0(\mathcal{C})$ is itself triangulated 
and the Yoneda embedding 
$H^0(\mathcal{C}) \rightarrow H^0(\pretriag(\mathcal{C}))$ is an equivalence. 
Fix its quasi-inverse 
$H^0(\pretriag(\mathcal{C})) \rightarrow H^0(\mathcal{C})$. We refer
to it as the \em convolution functor \rm and write 
$\left\{ E_i, q_{ij} \right\}$ for the convolution 
in $H^0(C)$ of the twisted complex $(E_i, q_{ij})$. We think of 
$\mathcal{C}$ as a DG-enhancement of 
the triangulated category $H^0(\mathcal{C})$ and of 
$\pretriag(\mathcal{C})$ as an enlargement of $\mathcal{C}$ to  
a bigger DG-enhancement which allows 
for the calculus of twisted complexes described below. 

Any one-sided twisted complex $(E_i, q_{ij})$ over $\mathcal{C}$ 
defines an ordinary differential complex 
\begin{align}
\label{eqn-complex-associated-to-a-twisted-complex}
\dots \xrightarrow{q_{i-2,i-1}} E_{i-1} \xrightarrow{q_{i-1,i}}
E_{i} \xrightarrow{q_{i,i+1}} E_{i+1} \xrightarrow{q_{i+1,i+2}} \dots
\end{align}
in $H^0(\mathcal{C})$. This is
because by the definition of a twisted complex all $q_{i,i+1}$ are
closed of degree $0$ and we have 
$q_{i,i+1} \circ q_{i-1,i} = (-1)^{i} dq_{i-1,i+1}$. 
It is well known that the data of the higher twisted differentials 
of $(E_i, q_{ij})$ defines  
a number of Postnikov systems for 
\eqref{eqn-complex-associated-to-a-twisted-complex} in $H^0(\mathcal{C})$ 
whose convolutions are all isomorphic to $\left\{ E_i, q_{ij} \right\}$. 
Below we describe this in detail for two- and three-term twisted
complexes. 

A two-term one-sided twisted complex concentrated in degrees $-1,0$ is
the data of
\begin{equation}
\label{eqn-two-term-twisted-complex}
\begin{tikzcd}[column sep={1.155cm}]
A
\ar{r}{f} 
& 
\underset{\degzero}{B}
\end{tikzcd}
\end{equation}
where $A,B \in \mathcal{C}$ and $f$ is a degree $0$ closed map in
$\mathcal{C}$. The corresponding complex in $H^0(\mathcal{C})$ is 
\begin{equation}
\label{eqn-two-term-ordinary-complex}
\begin{tikzcd}[column sep={1.155cm}]
A
\ar{r}{f} 
& 
B
\end{tikzcd}
\end{equation}
A Postnikov system for \eqref{eqn-two-term-ordinary-complex}
is an exact triangle incorporating $f$. The triangle
defined by \eqref{eqn-two-term-twisted-complex} is 
\begin{equation}
\label{eqn-two-term-Postnikov-system}
\begin{tikzcd}[column sep={1.155cm,between origins}, row sep={1.333cm,between origins}]
A
\ar{rr}{f} 
&&
B
\ar{ld}{k}
\\
& 
\{
A \xrightarrow{f} 
\underset{\degzero}{B}
\}
\ar[dashed]{lu}{l}
&
\end{tikzcd}
\end{equation}
where $l$ and $k$ are the images in $H^0(\mathcal{C})$ of the
following maps of twisted complexes:
\begin{equation*}
l:
\begin{tikzcd}[column sep={1.333cm,between origins}, row sep={1.333cm,between origins}]
 A
\ar{r}{f}
\ar[']{dr}{\id}
&
\underset{\degzero}{B}
\\
&
\underset{\degzero}{A}
\end{tikzcd}
\qquad \qquad k:
\begin{tikzcd}[column sep={1.333cm,between origins}, row sep={1.333cm,between origins}]
&
\underset{\degzero}{B}
\ar{d}{\id}
\\
A
\ar{r}{f}
&
\underset{\degzero}{B}
\end{tikzcd}
\end{equation*}

A three-term one-sided twisted complex concentrated in degrees $-2$, 
$-1$, $0$ is the data of
\begin{equation}
\label{eqn-three-term-twisted-complex}
\begin{tikzcd}[column sep={2cm}]
A
\ar{r}{f} 
\ar[dashed, bend left=20]{rr}{x}
& 
B
\ar{r}{g} 
&
\underset{\degzero}{C}
\end{tikzcd}
\end{equation}
where $A,B,C \in \mathcal{C}$, $f$ and $g$ are closed maps of degree
$0$ in $\mathcal{C}$ and $x$ is a degree $-1$ map in $\mathcal{C}$ 
such that that $dx=-g\circ f$. The corresponding 
complex in $H^0(\mathcal{C})$ is
\begin{equation}
\label{eqn-three-term-ordinary-complex}
\begin{tikzcd}[column sep={2cm}]
A
\ar{r}{f} 
& 
B
\ar{r}{g} 
&
C
\end{tikzcd}
\end{equation}
with the composition $g \circ f$ being zero in $H^0(\mathcal{C})$ as
it is explicitly a boundary $dx$ in $\mathcal{C}$. The data of the degree 
$-1$ map $x$ in $\mathcal{C}$ defines a right and a left Postnikov system 
for the complex \eqref{eqn-three-term-ordinary-complex} in $H^0(\mathcal{C})$:
\begin{defn}
\label{defn-induced-Postnikov-systems}
The \em right Postnikov system induced by the twisted complex 
\eqref{eqn-three-term-twisted-complex} \rm is 
\begin{equation}
\label{eqn-right-Postnikov-system-from-x}
\begin{tikzcd}[column sep={1.155cm,between origins}, row sep={1.333cm,between origins}]
A
\ar{rr}{f}
\ar[dotted]{drrr}[']{j}
& &
B
\ar{rr}{g}
\ar[phantom]{drr}[description, pos=0.45]{\star}
& &
C
\ar{dl}{h}
\\
& &
&
\{B \xrightarrow{g} \underset{\degzero}{C}\}
\ar[dashed]{ul}[']{i}
&
~
\end{tikzcd}
\end{equation}
where the maps $h$, $i$, $j$ are the images in $H^0(\C)$ of the
following maps of twisted complexes:
\begin{equation*}
h:
\begin{tikzcd}[column sep={1.333cm,between origins}, row sep={1.333cm,between origins}]
&
\underset{\degzero}{C}
\ar{d}{\id}
\\
B
\ar{r}{g}
&
\underset{\degzero}{C}
\end{tikzcd}
\qquad \qquad i:
\begin{tikzcd}[column sep={1.333cm,between origins}, row sep={1.333cm,between origins}]
B
\ar{r}{g}
\ar[']{dr}{\id}
&
\underset{\degzero}{C}
\\
&
\underset{\degzero}{B}
\end{tikzcd}
\qquad \qquad j:
\begin{tikzcd}
\underset{\degminusone}{A}
\ar[']{d}{f}
\ar{rd}{x}
&
\\
B
\ar{r}{g}
&
\underset{\degzero}{C}.
\end{tikzcd}
\end{equation*}
The \em left Postnikov system induced by the complex
\eqref{eqn-three-term-twisted-complex} \rm is
\begin{equation}
\label{eqn-left-Postnikov-system-from-x}
\begin{tikzcd}[column sep={1.155cm,between origins}, row sep={1.333cm,between origins}]
A
\ar{rr}{f}
& &
B
\ar{rr}{g}
\ar{dl}[']{k}
\ar[phantom]{dll}[description, pos=0.475]{\star}
& &
C
\\
~
& 
\ar[dashed]{ul}{l}
\{ A \xrightarrow{f} B \}
\ar{urrr}[']{m}
&
&
&
\end{tikzcd}
\end{equation}
where the maps $l$, $k$, $m$ are the images in $H^0(\C)$ of the respective maps:
\begin{equation*}
l:
\begin{tikzcd}[column sep={1.333cm,between origins}, row sep={1.333cm,between origins}]
A
\ar{r}{f}
\ar[']{dr}{\id}
&
\underset{\degzero}{B}
\\
&
\underset{\degzero}{A}
\end{tikzcd}
\qquad \qquad k:
\begin{tikzcd}[column sep={1.333cm,between origins}, row sep={1.333cm,between origins}]
&
\underset{\degzero}{B}
\ar{d}{\id}
\\
A
\ar{r}{f}
&
\underset{\degzero}{B}
\end{tikzcd}
\qquad \qquad m:
\begin{tikzcd}[column sep={1.333cm,between origins}, row sep={1.333cm,between origins}]
A
\ar{r}{f}
\ar[']{rd}{x}
&
\underset{\degzero}{B}
\ar{d}{g}
\\
&
\underset{\degzero}{C}.
\end{tikzcd}
\end{equation*}
\end{defn}

\begin{lemma}
\label{lemma-the-convolutions-of-induced-Postnikov-systems}
For any twisted complex \eqref{eqn-three-term-twisted-complex}
the convolutions of its left and right Postnikov systems are 
isomorphic in $H^0(\C)$ to the convolution of the twisted complex
itself.  
\end{lemma}
\begin{proof}
By definition the convolutions of  
\eqref{eqn-right-Postnikov-system-from-x} and 
\eqref{eqn-left-Postnikov-system-from-x} are 
$\cone (j)$ and $\cone (m)$, respectively. As we've seen, 
the cone of a map in $H^0(\mathcal{C)}$ is its convolution as 
a two-term twisted complex over $\mathcal{C}$. In case of $j$ and $m$, 
the objects of this twisted complex are themselves convolutions of
twisted complexes. The double convolution of a twisted complex of twisted 
complexes is isomorphic to the convolution of its total complex
\cite[\S2]{BondalKapranov-EnhancedTriangulatedCategories}. 
In case of both $j$ and $m$ these total complexes coincide with 
\eqref{eqn-three-term-twisted-complex}, whence the result. 
\end{proof}

The conceptual explanation for Lemma 
\ref{lemma-left-Postnikov-system-for-each-right-Postnikov-system-and-vice-versa}
is as follows. Any Postnikov system for a given complex in
$H^0(\mathcal{C})$ lifts (non-uniquely) to a twisted complex over
$\C$. In Lemma \ref{lemma-any-Postnikov-system-is-isomorphic-to-an-induced-one}
we prove this for three-term complexes. This twisted complex 
can then be used to induce a Postnikov system of any given type 
whose convolvution is isomorphic to the convolution of the original 
Postnikov system. In Lemma 
\ref{lemma-the-convolutions-of-induced-Postnikov-systems}
we prove this for three-term complexes.  The general case 
can be proved in a similar way but with a more convoluted notation.  

\begin{lemma}
\label{lemma-any-Postnikov-system-is-isomorphic-to-an-induced-one}
Any right or left Postnikov system for any differential complex \begin{equation}
\label{eqn-three-term-ordinary-complex-to-lift}
\begin{tikzcd}[column sep={1.333cm}]
A
\ar{r}{f} 
& 
B
\ar{r}{g} 
&
C
\end{tikzcd}
\end{equation}
in $H^0(\mathcal{C})$ is induced up to an isomorphism by some 
lift of \eqref{eqn-three-term-ordinary-complex-to-lift}
to a three-term twisted complex over $\mathcal{C}$. 
\end{lemma}
\begin{proof}
We prove the claim for left Postnikov systems, the proof for the right
ones is analogous. 
Any exact triangle incorporating $f$ is isomorphic to the exact
triangle \eqref{eqn-two-term-Postnikov-system}. Hence any left
Postnikov system for \eqref{eqn-three-term-ordinary-complex-to-lift}
is isomorphic to the one as in 
\eqref{eqn-left-Postnikov-system-from-x}
but with $m$ some unknown closed degree $0$ map in
$H^0(\mathcal{C})$. Since the convolution functor 
is an equivalence we can lift $m$ to some closed degree $0$ map 
of twisted complexes
\begin{equation}
\label{eqn-lift-m}
\begin{tikzcd}[column sep={1.333cm,between origins}, row sep={1.1cm,between origins}]
A
\ar{r}{f}
\ar[']{rd}{x'}
&
\underset{\degzero}{B}
\ar{d}{g'}
\\
&
\underset{\degzero}{C}. 
\end{tikzcd}
\end{equation}
We have $dx'+ g' \circ f = 0$ as the map is closed. By definition 
of a Postnikov system $m \circ k = g$ in $H^0(\C)$ and thus 
$g - g' = d \alpha$ for some degree $-1$ map $\alpha$. We then have
\begin{equation}
\label{eqn-modifying-x'-and-g'}
\begin{tikzcd}[column sep={1.333cm,between origins}, row sep={1.2cm,between origins}]
A
\ar{r}{f}
\ar[']{rd}{x'}
&
\underset{\degzero}{B}
\ar{d}{g'}
\\
&
\underset{\degzero}{C} 
\end{tikzcd}
\quad + \quad 
d\left(
\begin{tikzcd}[column sep={1.333cm,between origins}, row sep={1.2cm,between origins}]
A
\ar{r}{f}
&
\underset{\degzero}{B}
\ar{d}{\alpha}
\\
&
\underset{\degzero}{C} 
\end{tikzcd}
\right)
\quad = \quad 
\begin{tikzcd}[column sep={1.333cm,between origins}, row sep={1.2cm,between origins}]
A
\ar{r}{f}
\ar[']{rd}{x' - \alpha \circ f}
&
\underset{\degzero}{B}
\ar{d}{g}
\\
&
\underset{\degzero}{C}. 
\end{tikzcd}
\end{equation}
The right hand side of \eqref{eqn-modifying-x'-and-g'} is also 
a lift of $m$ and thus the left Postnikov system in question is 
induced by the three-term twisted complex lifting 
\eqref{eqn-three-term-ordinary-complex-to-lift}  
with $x=x' - \alpha \circ f$. 
\end{proof}

\subsection{$\mathbb{P}^n$-functors}
\label{section-p-functors}

Let $\A$ and $\B$ be enhanced triangulated categories.
As defined in \cite{Addington-NewDerivedSymmetriesOfSomeHyperkaehlerVarieties}
a \em split $\mathbb{P}^n$-functor \rm is a functor 
$F\colon \A \rightarrow \B$ which has left and right adjoints 
$ L, R \colon \B \rightarrow \A $ such that:
\begin{enumerate}
\item For some autoequivalence $H$ of $\A$ there exists an isomorphism   
\begin{align}
H^n \oplus H^{n-1} \oplus \dots \oplus H \oplus \id
\xrightarrow{\quad \gamma \quad} RF. 
\end{align}
\item (The strong monad condition) In the monad structure on $H^n \oplus
H^{n-1} \oplus \dots \oplus H \oplus \id$ induced by $\gamma^{-1}$
from the adjunction monad $RF$ the left multiplication by $H$ 
acts on 
\begin{align}
\label{eqn-H^n-dots-H}
H^n \oplus H^{n-1} \oplus \dots \oplus H
\end{align}
as an upper triangular matrix with $\id$'s on the main diagonal. 
Note that as such matrix is evidently invertible the resulting
endomorphism of \eqref{eqn-H^n-dots-H} is necessarily an isomorphism. 
\item (The weak adjoints condition) 
$R \simeq H^n L$. 
\end{enumerate}

Let $\psi$ be the $FHR \rightarrow FR$ 
component of $FRFR \xrightarrow{FR\trace - \trace FR} FR$ 
under the identification of
$FRFR$ with $FH^nR \oplus \dots \oplus FHR \oplus FR$ 
via $F\gamma R$. The \em $\mathbb{P}^n$-twist \rm $P_F$ was defined in 
\cite[\S3.3]{Addington-NewDerivedSymmetriesOfSomeHyperkaehlerVarieties}
as the convolution of 
\begin{align}
\label{eqn-FHR-FR-Id-complex}
FHR \xrightarrow{ \psi } FR \xrightarrow{\trace} \id 
\end{align}
given by a certain canonical right Postnikov system associated to it. 
Addington noted that such system is no longer unique but 
provided a canonical choice of one.

As mentioned in the introduction, the reason the notion introduced 
by Addington is best referred to as \em split \rm 
$\mathbb{P}^n$-functors is that the monad
$RF$ splits into a direct sum of $\id$ and powers of $H$. 
We propose the following more general notion 
of a $\mathbb{P}^n$-functor which allows $RF$ 
to be a repeated extension. 

\begin{defn}
Let $H$ be an endofunctor of $\A$. 
A \em cyclic extension of $\id$ by $H$ of degree $n$ \rm 
is a repeated extension $Q_n$ of the form
\begin{footnotesize}
\begin{equation}
\label{eqn-cyclic-extension-of-id-by-H-of-degree-n}
\begin{tikzcd}[column sep={0.35cm}]
\id 
\ar[phantom]{drr}[description, pos=0.45]{\star}
\ar{rr}{\iota_1}
& &
Q_1 
\ar{rr}{\iota_2}
\ar{ld}{\mu_1}
\ar[phantom]{drr}[description, pos=0.45]{\star}
& &
Q_2  
\ar{r}
\ar{ld}{\mu_2}
&
\dots
\ar{r}
&
Q_{n-2}
\ar[phantom]{drr}[description, pos=0.45]{\star}
\ar{rr}{\iota_{n-1}}
\ar{ld}
& 
& 
Q_{n-1}
\ar{ld}{\mu_{n-1}}
\ar{rr}{\iota_{n}}
& & 
Q_n.
\ar[phantom]{dll}[description, pos=0.45]{\star}
\ar{ld}{\mu_n}
\\
&
H
\ar[dashed]{lu}
& ~ &
H^2 
\ar[dashed]{lu}
\ar[dashed]{ll}
& ~ & 
\dots 
\ar[dashed]{ll}
\ar[dashed]{lu}
&
&
H^{n-1}
\ar[dashed]{ll}
\ar[dashed]{lu}
&~& 
H^n
\ar[dashed]{lu}
\ar[dashed]{ll}
& 
\end{tikzcd}
\end{equation}
\end{footnotesize}
Here all starred triangles
are exact, all the remaining triangles are commutative, and all
the dashed arrows denote maps of degree $1$. We further write
$\iota$ for the map $\id \xrightarrow{\iota_n \circ \dots \circ
\iota_1} Q_n$. 
\end{defn}

Equivalently, $Q_n$ is the convolution of a one-sided
twisted complex of the form
\begin{align}
H^n[-n] \rightarrow
H^{n-1}[-(n-1)] \rightarrow \dots \rightarrow 
H^2[-2] \rightarrow
H[-1] \rightarrow
\underset{\degzero}{\id} 
\end{align}
with 
arbitrary differentials. The maps $\id 
\xrightarrow{\iota} Q_n$ and $Q_n \xrightarrow{\mu_n} H^n$
are the inclusion of the degree $0$ term and the 
projection on the degree $-n$ term, respectively. 

\begin{defn}
A \em $\mathbb{P}^n$-functor \rm is a functor 
$F\colon \A \rightarrow \B $
with left and right adjoints $L, R \colon \B \rightarrow \A$
such that 
\begin{enumerate}
\item 
There exists an isomorphism 
$$ Q_n \xrightarrow{\gamma} RF $$
where $Q_n$ is a cyclic extension of $\id_\A$ by an autoequivalence $H$ 
of $\A$ with $H(\krn F) = \krn F$. Moreover, this isomorphism
intertwines $\id \xrightarrow{\action} RF$ and 
$\id \xrightarrow{\iota} Q_n$. 

Note that as $F \xrightarrow{F\action} FRF$ is a retract, so is 
$F\iota$. Hence the exact triangle $FR \to FQ_1R \to FHR$ is also 
split. Choose any splitting $FHR \to FQ_1R$ and denote
by $\phi$ the composition
\begin{equation}
FHR \xrightarrow{} FQ_1R 
\xrightarrow{\iota_n\circ\ldots \circ \iota_2}
FQ_nR \xrightarrow{ F \gamma R} FRFR.
\end{equation}
Define the map $FHR \xrightarrow{\psi} FR$ to be the 
composition 
$$FHR \xrightarrow{\phi} FRFR \xrightarrow{FR\trace - \trace FR} FR.$$
Note that any choice of the splitting $FHR\to FQ_1R$ in the definition of $\phi$ 
will produce the same map $\psi$, since the composition
$$FR \xrightarrow{F\action R} FRFR \xrightarrow{FR\trace - \trace FR}
FR$$
is zero.

\item (The monad condition) The following composition is an isomorphism:
\begin{align}
\label{eqn-the-monad-condition-map}
FHQ_{n-1} \xrightarrow{FH \iota_{n-1}} 
FHRF \xrightarrow{\psi F} FRF \xrightarrow{F\kappa} FC[1].
\end{align}
Here $C$ is the spherical cotwist of $F$ defined by an exact triangle
$$ C \rightarrow \id \xrightarrow{\action} RF \xrightarrow{\kappa}
C[1].$$ 
\item (The adjoints condition) The following composition is an isomorphism: 
\begin{align}
FR \xrightarrow{FR \action } FRFL \xrightarrow{\mu_n L} FH^n L. 
\end{align}

\item (The highest degree term condition) There is an isomorphism that makes the diagram commute:
\begin{equation*}
\begin{footnotesize}
\begin{tikzcd}[column sep="1.6cm"]
FHQ_{n-1}L
\ar{r}{FH\iota_n L}
\ar[equals]{d}
&
FHRFL
\ar{r}{\psi FL}
&
FRFL
\ar{r}{F\mu_nL}
&
FH^nL
\ar[dashed]{d}
\\
FHQ_{n-1}L
\ar{r}{FH\iota_nL}
&
FHRFL
\ar{r}{FHR\psi'}
&
FHRFH'L
\ar{r}{FH\mu_n H'L}
&
FHH^nH'L,
\end{tikzcd}
\end{footnotesize}
\end{equation*}
where $H'$ is the left adjoint of $H$ and 
$\psi'\colon FL \to FH'L$ is the left dual to 
$\psi\colon FHR \to FR$.
\end{enumerate}
\end{defn}

In the split case treated by Addington the objects 
$FHQ_{n-1}$ and $FC[1]$ are both isomorphic to 
$$ FH^n \oplus \dots \dots \oplus FH. $$ 
The map \eqref{eqn-the-monad-condition-map} 
is the image under $F$ of the 
left multiplication by $H$ in the $RF$ monad structure 
minus a strictly upper triangular matrix. 
Our monad condition asks for \eqref{eqn-the-monad-condition-map} 
to be invertible, while the one in  
\cite{Addington-NewDerivedSymmetriesOfSomeHyperkaehlerVarieties}
asks for the left multiplication by $H$ to be upper triangular with $\id$'s  
on the main diagonal. The precise non-split analogue of this would be 
requesting the map \eqref{eqn-the-monad-condition-map} to come from 
a one-sided map of twisted complexes whose vertical arrows are homotopy 
equivalences. This stronger condition implies our highest degree term condition
and 
implies that $RF \xrightarrow{R{\action}L} RFLF \xrightarrow{\mu_n LF} H^nLF$
is an isomorphism \cite[Lemmas 5.16 and 5.13]{AnnoLogvinenko-PFunctors}. 
That, in turn, means that the existence of any isomorphism $FR \simeq FH^nL$ 
implies our adjoints condition above
\cite[Prop.~5.14]{AnnoLogvinenko-PFunctors}. 
Thus, even the non-split analogue of the definition of a $\mathbb{P}^n$-functor 
in \cite{Addington-NewDerivedSymmetriesOfSomeHyperkaehlerVarieties} implies 
our definition. 

\begin{defn}
The \em $\mathbb{P}$-twist \rm $P_F$ of a $\mathbb{P}^n$-functor $F$ is
the unique convolution of the complex  
\begin{align}
FHR \xrightarrow{\psi} FR \xrightarrow{\trace} \id. 
\end{align}
\end{defn}

The uniqueness of the convolution is the main result of this paper, 
see Theorem \ref{theorem-uniqueness-of-FG-FR-Id-convolutions-dg-version}
and Theorem \ref{theorem-uniqueness-of-FG-FR-Id-convolutions-triangulated-version}. An upcoming paper \cite{AnnoLogvinenko-PFunctors}
proves that this $\mathbb{P}$-twist is indeed an autoequivalence of $\B$. 

\section{Uniqueness of $\mathbb{P}$-twists}
\subsection{An approach via triangulated categories}
\label{section-an-approach-via-triangulated-categories}
Let $Z$ and $X$ be separated schemes of finite type over a field $k$. 
We work with Fourier-Mukai kernels using the functorial notation: 
e.g. for any Fourier-Mukai kernels 
$F \in D(Z \times X)$ and $G \in D(X \times Z)$ of exact functors 
$D(Z) \xrightarrow{f} D(X)$ and $D(X) \xrightarrow{g} D(Z)$ 
we write $FG$ for the Fourier-Mukai kernel of $f \circ g$ given by the
standard Fourier-Mukai kernel composition:
$$ \pi_{13 *}(\pi_{12}^* G \ldertimes \pi_{23}^* F) \in D(X \times X). $$
Here
$\pi_{ij}$ are projections from $X \times Z \times X$ to 
the corresponding partial fiber products. We further 
write $\id_Z \in D(Z \times Z)$ and $\id_X \in D(X \times X)$ for 
the structure sheafs of the diagonals.

Let $F \in D(Z \times X)$ and $R \in D(X \times Z)$ be 
Fourier-Mukai kernels and let maps $FR \xrightarrow{\trace} \id_X$ and 
$\id_Z \xrightarrow{\action} RF$ define a $2$-categorical adjunction of
$F$ and $R$, i.e. the following compositions are identity maps:
$$ F \xrightarrow{F\action} FRF \xrightarrow{\trace F} F,$$
$$ R \xrightarrow{\action R} RFR \xrightarrow{R \trace } R.$$
In other words, consider adjoint exact functors 
$(f,r)\colon D(Z) \rightleftarrows D(X)$
with a fixed lift to $2$-categorically adjoint Fourier-Mukai kernels
$(F,R)$. Let $G$ be a Fourier-Mukai kernel of an 
exact functor $g: D(X) \rightarrow D(Z)$. 

\begin{theorem}
\label{theorem-uniqueness-of-FG-FR-Id-convolutions-triangulated-version}
For any $FG \xrightarrow{f} FR$ with $\trace \circ f = 0$
all convolutions of the following complex are isomorphic:
\begin{align}
\label{eqn-FG-FR-Id-complex}
FG \xrightarrow{f} FR \xrightarrow{\trace} \id_X.
\end{align}
\end{theorem}
\begin{proof}
By Lemma
\ref{lemma-left-Postnikov-system-for-each-right-Postnikov-system-and-vice-versa}
it suffices to show that the convolutions of all right Postnikov
systems associated to  
\eqref{eqn-FG-FR-Id-complex}
are isomorphic, since for any left Postnikov system there
exists a right Postnikov system with an isomorphic convolution. Then  
by Lemma \ref{lemma-uniqueness-of-convolutions-for-a-two-step-complex}
it suffices to show that the natural map 
$$ 
\homm^{-1}_{D(X \times X)}(FG,FR) 
\xrightarrow{\trace \circ (-)}
\homm^{-1}_{D(X \times X)}(FG,\id_X) 
$$ 
is surjective. The idea is simple: by the $2$-categorical adjunction of
$F$ and $R$ it suffices to show the surjectivity of 
$$ 
\homm^{-1}_{D(X \times Z)}(G,RFR) 
\xrightarrow{R\trace \circ (-)}
\homm^{-1}_{D(X \times Z)}(G,R), 
$$ 
but this is trivial: $RFR \xrightarrow{R \trace} R$
has a right quasi-inverse $R \xrightarrow{\action R} RFR$.

Indeed, let $\phi \in \homm^{-1}_{D(X \times X)}(FG,\id_X)$ be any element.
Let 
$$\psi \in \homm^{-1}_{D(X \times X)}(FG,FR)$$
be the composition 
$$ FG \xrightarrow{F\action G} FRFG \xrightarrow{FR \phi} FR. $$
Then $\trace \circ \psi$ is the composition 
$$ FG \xrightarrow{F\action G} FRFG \xrightarrow{FR\phi} FR
\xrightarrow{\trace} \id_X.$$
Since the composition of Fourier-Mukai kernels is functorial 
the following two compositions are equal:
$$ FRFG \xrightarrow{FR\phi} FR \xrightarrow{\trace} \id_X \quad\text{ and }\quad 
 FRFG \xrightarrow{\trace FG} FG \xrightarrow{\phi} \id_X. $$
Thus $\trace \circ \psi$ equals the composition 
$$ FG \xrightarrow{F\action G} FRFG \xrightarrow{\trace FG} FG
\xrightarrow{\phi} \id_X$$
which is just $\phi$ since $(\trace FG) \circ (F\action G) = \id$. 
We conclude that $\trace \circ (-)$ is surjective as desired. 
\end{proof}

\subsection{An approach via DG-enhancements}
\label{section-an-approach-via-DG-enhancements}

Let $\A$ and $\B$ be two \em enhanced triangulated categories\rm. 
In other words, $\A$ and $\B$ are quasi-equivalence classes 
of pretriangulated $DG$-categories. 
The underlying triangulated categories are $H^0(\A)$ and $H^0(\B)$. 
Let $D(\AbimB)$ be the derived category of $\AbimB$-bimodules. 
The DG-enhanceable exact functors $H^0(\A) \rightarrow H^0(\B)$ are in 
one-to-one correspondence with the isomorphism classes in
$D^{\Bquasirep}(\AbimB)$, the full subcategory of $D(\AbimB)$ 
consisting of $\B$-quasi-representable bimodules
\cite{Toen-TheHomotopyTheoryOfDGCategoriesAndDerivedMoritaTheory}. 
If the underlying
triangulated categories are Karoubi-complete, we can use the Morita
framework where $\A$ and $\B$ are Morita equivalence classes of 
small DG-categories. The underlying triangulated categories are
the full subcategories $D_c(\A)$ and $D_c(\B)$ of the compact 
objects in $D(\A)$ and $D(\B)$. The DG-enhanceable exact functors are  
in one-to-one correspondence with the isomorphism classes in
$D^{\Bperf}(\AbimB)$, the full subcategory of $D(\AbimB)$ consisting
of $\B$-perfect bimodules
\cite{Toen-TheHomotopyTheoryOfDGCategoriesAndDerivedMoritaTheory}. 
Either way, this shows that to make
the results of this section applicable to any pair of adjoint
DG-enhanceable exact functors between two enhanced triangulated 
categories it suffices to work with homotopy adjoint DG-bimodules. 

Let $\A$ and $\B$ be two small DG categories. Let $\AmodbarB$,
$\BmodbarA$, $\AmodbarA$ and $\BmodbarB$ be the bar categories of 
$\AbimB$-, $\BbimA$-, $\AbimA$- and $\BbimB$-bimodules \cite{AnnoLogvinenko-BarCategoryOfModulesAndHomotopyAdjunctionForTensorFunctors}. 
These could be replaced by any other DG enhancements of 
the derived categories of bimodules equipped with (homotopy) 
unital tensor bifunctors 
$(-)\otimes_\A (-)$ and $(-)\otimes_\B (-)$ which descend to the bifunctors
$(-)\ldertimes_\A (-)$ and $(-)\ldertimes_\B (-)$ between 
the derived categories. For example, one can take $h$-projective or
$h$-injective enhancements. The advantage of bar categories is that 
any adjunction of DG-enhanceable functors can be lifted to 
a pair of homotopy adjoint bimodules described in the next
paragraph, cf. 
\cite[\S 5.2]{AnnoLogvinenko-BarCategoryOfModulesAndHomotopyAdjunctionForTensorFunctors}

Let $M \in \AmodbarB$ and $N \in \BmodbarA$ be 
\em homotopy adjoint\rm, that is --- there exist maps 
\begin{align*}
\trace: N\otimes_\A M \to \B
& \qquad \qquad
\action: \A \to M \otimes_\B N
\end{align*}
in $\AmodbarA$ and $\BmodbarB$ such that 
\begin{align}
\label{eqn-F-FRF-F-composition}
M \xrightarrow{\action \otimes \id} M \otimes_\B N \otimes_\A M 
\xrightarrow{\id \otimes \trace} M \\
\label{eqn-R-RFR-R-composition}
N \xrightarrow{\id \otimes \action} N \otimes_\A M \otimes_\B N
\xrightarrow{\trace \otimes \id} N
\end{align}
are homotopic to $\id_M$ and $\id_N$. Thus 
there exists a degree $-1$ map $\zeta: M \to M$ such that 
$\eqref{eqn-F-FRF-F-composition} = \id_M + d\zeta$.
  
Let $X \in \BmodbarA$ and let $X\otimes_\A M \xrightarrow{f} N\otimes_\A M$ 
be any map such that the following is a differential complex in
$D(\BbimB) \simeq H^0(\BmodbarB)$:
\begin{equation}
\label{eqn-ordinary-complex-XM-NM-B}
\begin{tikzcd}[column sep={3cm}]
X\otimes_\A M
\ar{r}{f} 
& 
N\otimes_\A M
\ar{r}{\trace} 
&
\B.
\end{tikzcd}
\end{equation}

\begin{prps}
\label{prps-any-two-lifts-of-XM-NM-B-are-isomorphic}
Any two lifts of \eqref{eqn-ordinary-complex-XM-NM-B} to a twisted
complex over $\BmodbarB$ are homotopy equivalent. 
\end{prps}
\begin{proof}
Any lift of \eqref{eqn-ordinary-complex-XM-NM-B} in $D(\BbimB)$ 
to a twisted complex over $\BmodbarB$ is readily seen to be homotopy 
equivalent to a one which lifts $f$ to $f$ and $\trace$ to $\trace$.
The latter is simply a choice of the degree $-1$ map 
$h\colon X\otimes_\A M \rightarrow \B$ with $\trace \circ f + dh = 0$. 
Let $h_1$ and $h_2$ be any two such maps.  
Define $\xi$ to be the composition
\begin{equation*}
X\otimes_\A M
\xrightarrow{\id\otimes\action\otimes\id}
X\otimes_\A M \otimes_\B  N\otimes_\A M
\xrightarrow{(h_1-h_2)\otimes \id^{\otimes 2}}
 N\otimes_\A M.
\end{equation*}
Then $d\xi=0$. Consider the following diagram:
\begin{equation}
\label{eqn-XM-to-B-commutative-diagram}
\begin{tikzcd}[column sep={1.55cm}]
X \otimes_\A M
\ar{r}{\id \otimes \action \otimes \id}
\ar[']{rd}{\id}
&
X\otimes_\A M \otimes_\B  N\otimes_\A M
\ar{r}{(h_1-h_2)\otimes \id^{\otimes 2}}
\ar{d}{\id^{\otimes 2} \otimes \trace}
&
 N\otimes_\A M
\ar{d}{\trace}
\\
&
X \otimes_\A M
\ar{r}{h_1-h_2}
&
\B.
\end{tikzcd}
\end{equation}
It descends to a commutative diagram 
in $D(\BbimB)$, thus it commutes up to homotopy $\BmodbarB$. 
Let $\eta$ be the homotopy up to which it commutes, so that
$d\eta = \trace\circ\xi - h_1 + h_2$. Then the following
are two mutually inverse isomorphisms of twisted complexes:
\begin{equation*}
\begin{tikzcd}[column sep={2cm}]
X \otimes_\A M
\ar[equal]{d}
\ar[dashed, bend left=20]{rr}[']{h_1}
\ar{r}{f}
\ar[']{rd}[description]{\xi}
\ar{rrd}[description, pos=0.6]{-\eta}
&
N \otimes_\A M
\ar[equal]{d}
\ar{r}{\trace}
&
\B
\ar[equal]{d}
\\
X \otimes_\A M
\ar[dashed, bend right=20]{rr}{h_2}
\ar[']{r}{f}
&
N \otimes_\A M
\ar[']{r}{\trace}
&
\B
\end{tikzcd}
\end{equation*}
\begin{equation*}
\begin{tikzcd}[column sep={2cm}]
X \otimes_\A M
\ar[equal]{d}
\ar[dashed, bend left=20]{rr}[']{h_2}
\ar{r}{f}
\ar[']{rd}[description]{-\xi}
\ar{rrd}[description, pos=0.6]{\eta}
&
N \otimes_\A M
\ar[equal]{d}
\ar{r}{\trace}
&
\B
\ar[equal]{d}
\\
X \otimes_\A M
\ar[dashed, bend right=20]{rr}{h_1}
\ar[']{r}{f}
&
N \otimes_\A M
\ar[']{r}{\trace}
&
\B.
\end{tikzcd}
\end{equation*}
\end{proof}

\begin{theorem}
\label{theorem-uniqueness-of-FG-FR-Id-convolutions-dg-version}
Let $\A$ and $\B$ be enhanced triangulated categories. Let 
$F$ be an exact functor $\A \rightarrow \B$ with a right adjoint $R$. 
Let $\trace\colon FR \rightarrow \id_\B$ be the adjunction counit. 
Let $G$ be any exact functor $\B \rightarrow \A$ and  
$f\colon FG \rightarrow FR$ any natural transformation 
with $f \circ \trace = 0$. Finally, assume these are all
DG-enhanceable.

Then all convolutions of the following three-term complex are
isomorphic:
\begin{align}
\label{eqn-ordinary-complex-FG-FR-Id-enhanced-setting}
FG \xrightarrow{f} FR \xrightarrow{\trace} \id_\B. 
\end{align}
\end{theorem}
\begin{proof}
As per the beginning of this section we can lift $F$ and $R$ to 
a pair of homotopy adjoint  bimodules $M$ and $N$ and 
we can lift $G$ to an bimodule $X$. Then by 
Prop.~\ref{prps-any-two-lifts-of-XM-NM-B-are-isomorphic}
any two lifts of \eqref{eqn-ordinary-complex-FG-FR-Id-enhanced-setting}
to a twisted complex are homotopy equivalent. By Lemmas 
\ref{lemma-the-convolutions-of-induced-Postnikov-systems}
and 
\ref{lemma-any-Postnikov-system-is-isomorphic-to-an-induced-one}
every convolution of \eqref{eqn-ordinary-complex-FG-FR-Id-enhanced-setting}
is isomorphic in $D(\BbimB)$ to the convolution of some twisted complex 
lifting it. It follows that all convolutions of 
\eqref{eqn-ordinary-complex-FG-FR-Id-enhanced-setting}
are isomorphic. 
\end{proof}

\bibliography{references}
\bibliographystyle{amsalpha}
\vspace{0.5cm}
\end{document}